\documentclass[12pt]{article}
\usepackage[margin=1in]{geometry}
\usepackage[usenames]{color}
\usepackage[colorlinks=true,
linkcolor=webgreen,
filecolor=webbrown,
citecolor=webgreen]{hyperref}

\definecolor{webgreen}{rgb}{0,.5,0}
\definecolor{webbrown}{rgb}{.6,0,0}

\usepackage{color}

\usepackage{url}
\usepackage{graphicx}
\usepackage{subcaption}

\usepackage[T1]{fontenc}

\usepackage{amsmath}
\usepackage{amssymb}
\usepackage{amsthm}
\usepackage{mathrsfs}

\newcommand{\seqnum}[1]{\href{https://oeis.org/#1}{\rm \underline{#1}}}

\theoremstyle{plain}
\newtheorem{theorem}{Theorem}

\theoremstyle{definition}

\newtheorem{example}[theorem]{Example}

\newtheorem{problem}{Problem}

\theoremstyle{remark}

\title{Congruence properties of combinatorial sequences via Walnut
and the Rowland--Yassawi--Zeilberger automaton}
\author{Narad Rampersad\footnote{
Department of Math/Stats,
University of Winnipeg,
515 Portage Ave.,
Winnipeg, MB, R3B 2E9
Canada; {\tt narad.rampersad@gmail.com}.}\quad and\quad Jeffrey
Shallit\footnote{School of Computer Science, University of Waterloo,
Waterloo, ON  N2L 3G1, Canada; {\tt shallit@uwaterloo.ca}.}}

\begin{document}
\maketitle
\begin{abstract}
Certain famous combinatorial sequences, such as the Catalan numbers
and the Motzkin numbers, when taken modulo a prime power, can be computed by
finite automata.  Many theorems about such sequences can therefore be
proved using Walnut, which is an implementation of a decision procedure
for proving various properties of automatic sequences.  In this paper we
explore some results (old and new) that can be proved using this method.
\end{abstract}

\section{Introduction}
We study the properties of two famous combinatorial sequences.
For $n \geq 0$, let $$C_n = \frac{1}{n+1}\binom{2n}{n}$$ denote the
\emph{$n$-th Catalan number} and let $$M_n = \sum_{k=0}^{\lfloor n/2
  \rfloor}\binom{n}{2k}C_k$$ denote the \emph{$n$-th Motzkin number}.
For more about the Catalan numbers, see
\cite{Stanley}.

Many authors have studied congruence properties of these and other
sequences modulo primes $p$ or prime powers $p^\alpha$.
Notably, Alter and Kubota \cite{AK73} studied the Catalan numbers modulo $p$,
and Deutsch and Sagan \cite{DS06} studied many sequences,
including the Catalan numbers, Motzkin numbers, Central Delannoy numbers,
Ap\'ery numbers, etc., modulo certain prime powers.
Eu, Liu, and Yeh \cite{ELY08} studied the Catalan and Motzkin numbers
modulo $4$ and $8$, and Krattenthaler and M\"uller \cite{KM18} studied
the Motzkin numbers and related sequences modulo powers of $2$.
Rowland and Yassawi \cite{RY15} and Rowland and Zeilberger \cite{RZ14}
gave different methods to compute finite automata that compute the
sequences $(C_n \bmod p^\alpha)_{n \geq 0}$ and $(M_n \bmod
p^\alpha)_{n \geq 0}$ (and many other similar sequences),
where $p^\alpha$ is a prime power.  Rowland and
Zeilberger provide a number of these automata for different $p^\alpha$
at the website
\begin{center}
  \url{https://sites.math.rutgers.edu/~zeilberg/mamarim/mamarimhtml/meta.html}
\end{center}
along with the Maple code used to compute them.
We use some of these automata, along with the program Walnut
\cite{Walnut}, available at the website
\begin{center}
\url{https://cs.uwaterloo.ca/~shallit/walnut.html}
\end{center}
to study properties of these sequences.

We use the Rowland--Zeilberger algorithm (and Walnut) as a black-box,
so we do not discuss the theory behind it.  We just mention that
this algorithm applies to any sequence of numbers that can be defined
as the \emph{constant term} of $[P(x)]^nQ(x)$, where $P$ and $Q$ are
\emph{Laurent polynomials}.  In particular, the $n$-th Catalan number
is the constant term of $(1/x+2+x)^n(1-x)$ and the $n$-th Motzkin number
is the constant term of $(1/x+1+x)^n(1-x^2)$.

Note that the automata produced by this method read their input in
least-significant-digit-first format.  All of the automata in this
paper therefore also follow this convention.  We use the notation
$(n)_k$ to denote the base-$k$ representation of
$n$ in the lsd-first format.

Burns has posted several manuscripts to the arXiv
\cite{Bur_A,Bur_B,Bur_C,Bur_D,Bur_E,Bur_F} in which he investigates
various properties of the Catalan and Motzkin numbers modulo primes $p$
by analyzing structural properties of automata computed using the
Rowland--Yassawi algorithm.  This paper takes a similar approach, but
we use Walnut to simplify/automate much of the analysis.

\section{Motzkin numbers}
Deutsch and Sagan \cite{DS06} gave a characterization of ${\bf m}_2 = (M_n \bmod 2)_{n \geq
  0}$ that involves the Thue--Morse sequence $${\bf t} = (t_n)_{n \geq
  0} = (0,1,1,0,1,0,0,1,\ldots).$$  Let $${\bf c} = (c_n)_{n \geq 0} =
(1,3,4,5,7,\ldots)$$ denote the starting positions of the ``runs'' in
${\bf t}$, excluding the first run (which, of course, starts at position
$0$).

\begin{theorem}[Deutsch and Sagan] The Motzkin number $M_n$ is even if
  and only if either $n \in
  4{\bf c}-2$ or $n \in 4{\bf c}-1$.
\end{theorem}

\begin{proof}
  We can prove this result using Walnut.  The Rowland--Zeilberger algorithm
  produces the algorithm in Figure~\ref{MOT2}, which, when fed with $(n)_2$,\
  outputs $M_n \bmod 2$.

  \begin{figure}[htb]
    \centering
    \includegraphics[scale=0.75]{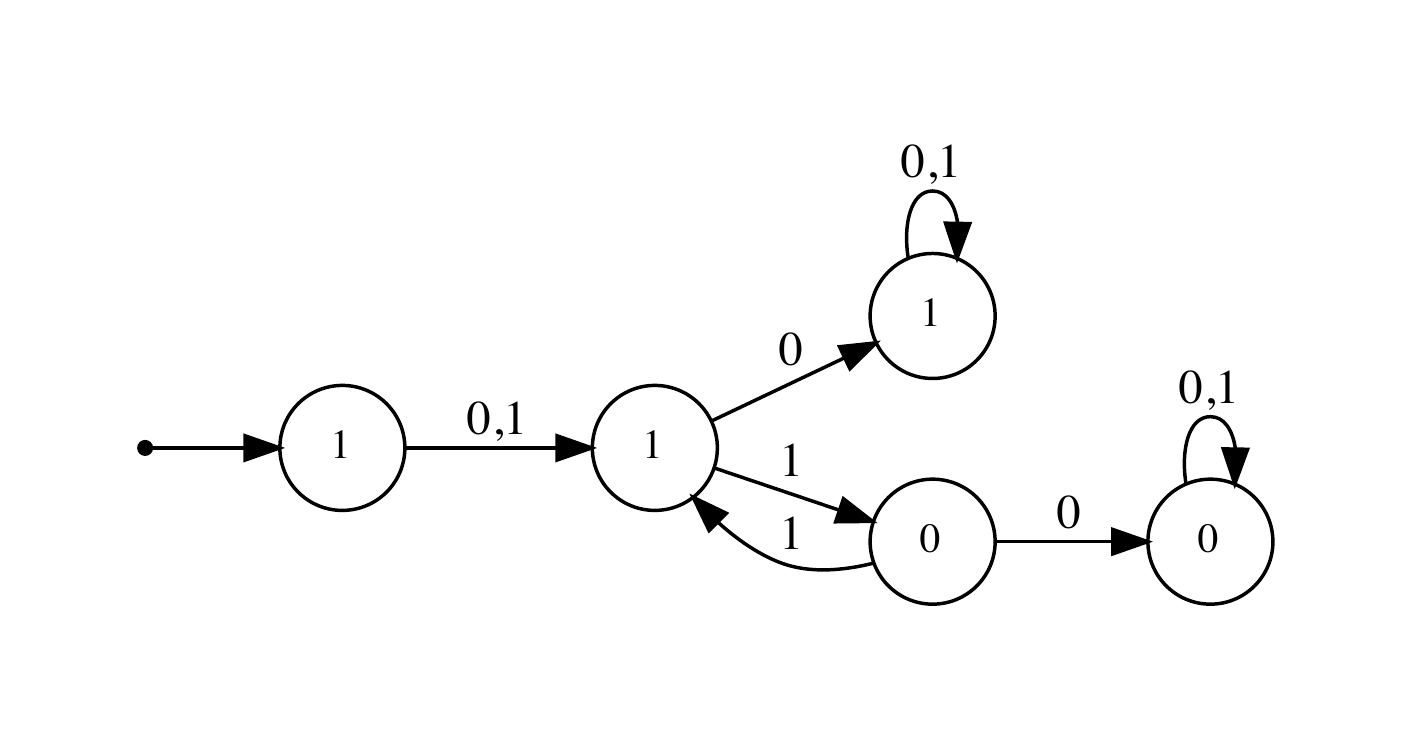}
    \caption{Automaton for $M_n \bmod 2$}\label{MOT2}
  \end{figure}

Next we use Walnut to construct an
automaton for the sequence ${\bf c}$.  The commands
\begin{verbatim}
def tm_blocks "?lsd_2 n>=1 & (At t<n => T_lsd[i+t]=T_lsd[i]) &
    T_lsd[i+n]!=T_lsd[i] & (i=0|T_lsd[i-1]!=T_lsd[i])":
def tm_block_start "?lsd_2 i>=1 & ($tm_blocks(i,1)|$tm_blocks(i,2))":
\end{verbatim}
produce the automaton \texttt{tm\_block\_start} given in
Figure~\ref{tm_block_start}, which computes ${\bf c}$.
We see that the elements of ${\bf c}$ are
$$ \{ m4^k : m \text{ is odd and } k \geq 0 \}.$$

  \begin{figure}[htb]
    \centering
    \includegraphics[scale=0.75]{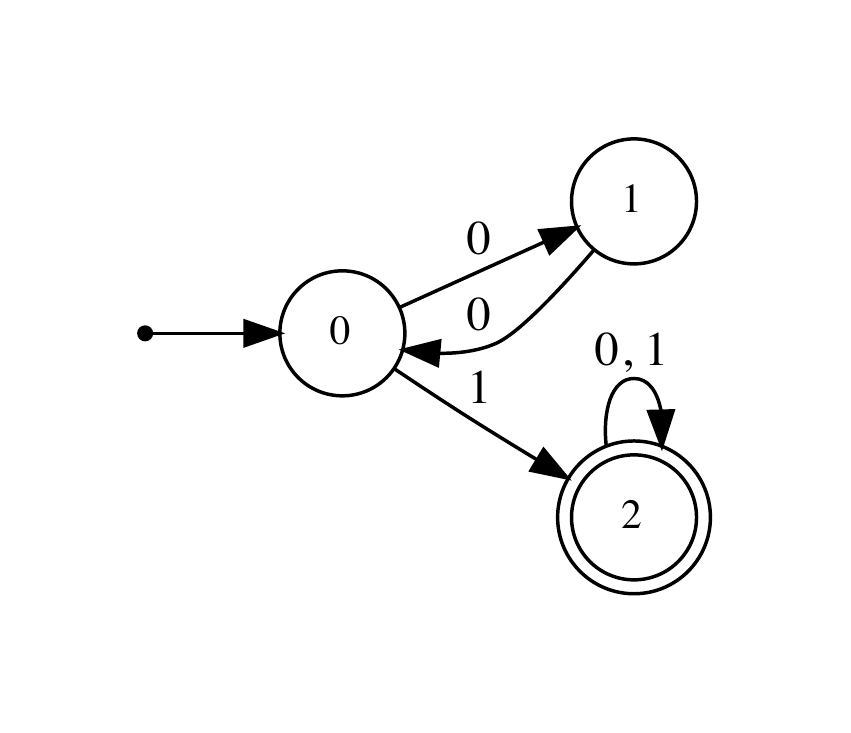}
    \caption{Automaton for starting positions of ``runs'' in {\bf t}}\label{tm_block_start}
  \end{figure}

To complete the proof of the theorem, it suffices to execute the
Walnut command
\begin{verbatim}
eval even_mot "?lsd_2 An (MOT2[n]=@0 <=> Ei $tm_block_start(i) &
    (n+2=4*i | n+1=4*i))":
\end{verbatim}
which produces the output ``TRUE''.
\end{proof}

Deutsch and Sagan also characterized ${\bf m}_3 = (M_n \bmod 3)_{n \geq
  0}$:
  
\begin{theorem}{(Deutsch and Sagan)}
The Motzkin number $M_n$ satisfies
\[
M_n \equiv_3
\begin{cases}
1, & \text{ if either } (n)_3 = 0w, w \in \{0,1\}^* \text{ or }
(n+2)_3 = 0w, w \in \{0,1\}^*, \\
2, & \text{ if } (n+1)_3 = 0w, w \in \{0,1\}^*, \\
0, & \text{ otherwise.}
\end{cases}
\]
\end{theorem}

This can also be obtained directly from the automaton for ${\bf m}_3$.
If we examine ${\bf m}_5 = (M_n \bmod 5)_{n \geq 0}$, however, we
discover that its behaviour is very different
from that of ${\bf m}_3$.  Deutsch and Sagan determined the positions
of the $0$'s in ${\bf m}_5$.

\begin{theorem}{(Deutsch and Sagan)}
The Motzkin number $M_n$ is divisible by $5$ if and only if
$n$ is of the form
\[ (5i+1)5^{2j}-2, \quad (5i+2)5^{2j-1}-1 \quad
(5i+3)5^{2j-1}-2 \quad (5i+4)5^{2j}-1 . \]
\end{theorem}

\begin{proof}
The Rowland--Zeilberger algorithm gives the automaton in Figure~\ref{MOT5}, which fully characterizes ${\bf m}_5$.

  \begin{figure}[htb]
    \centering
    \includegraphics[scale=0.60]{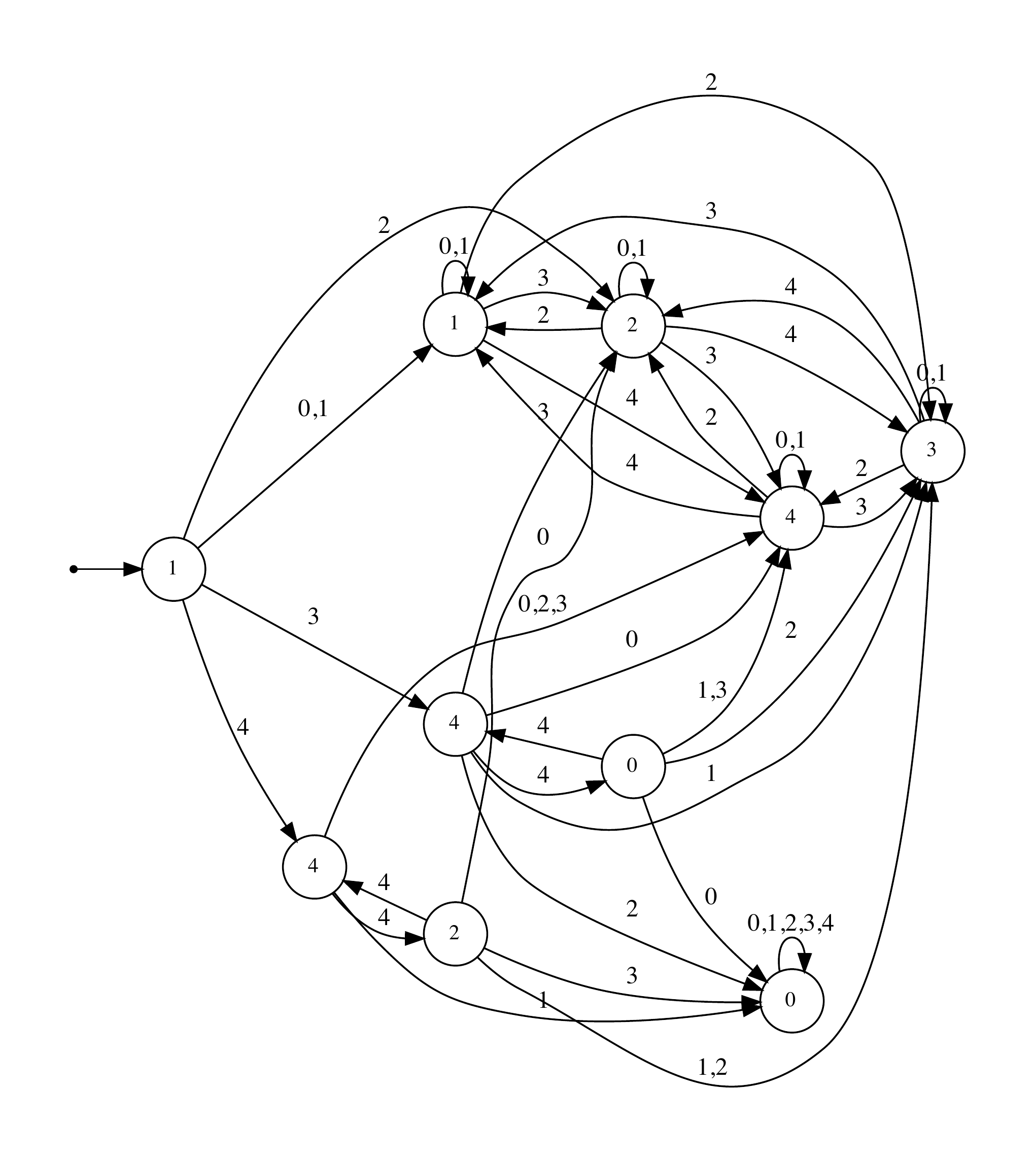}
    \caption{Automaton for $M_n \bmod 5$}\label{MOT5}
  \end{figure}

The Walnut command
\begin{verbatim}
eval mot5mod0 "?lsd_5 MOT5[n]=@0":
\end{verbatim}
produces the automaton in Figure~\ref{mot5mod0}, from which
one easily derives the result.
  \begin{figure}[htb]
    \centering
    \includegraphics[scale=0.75]{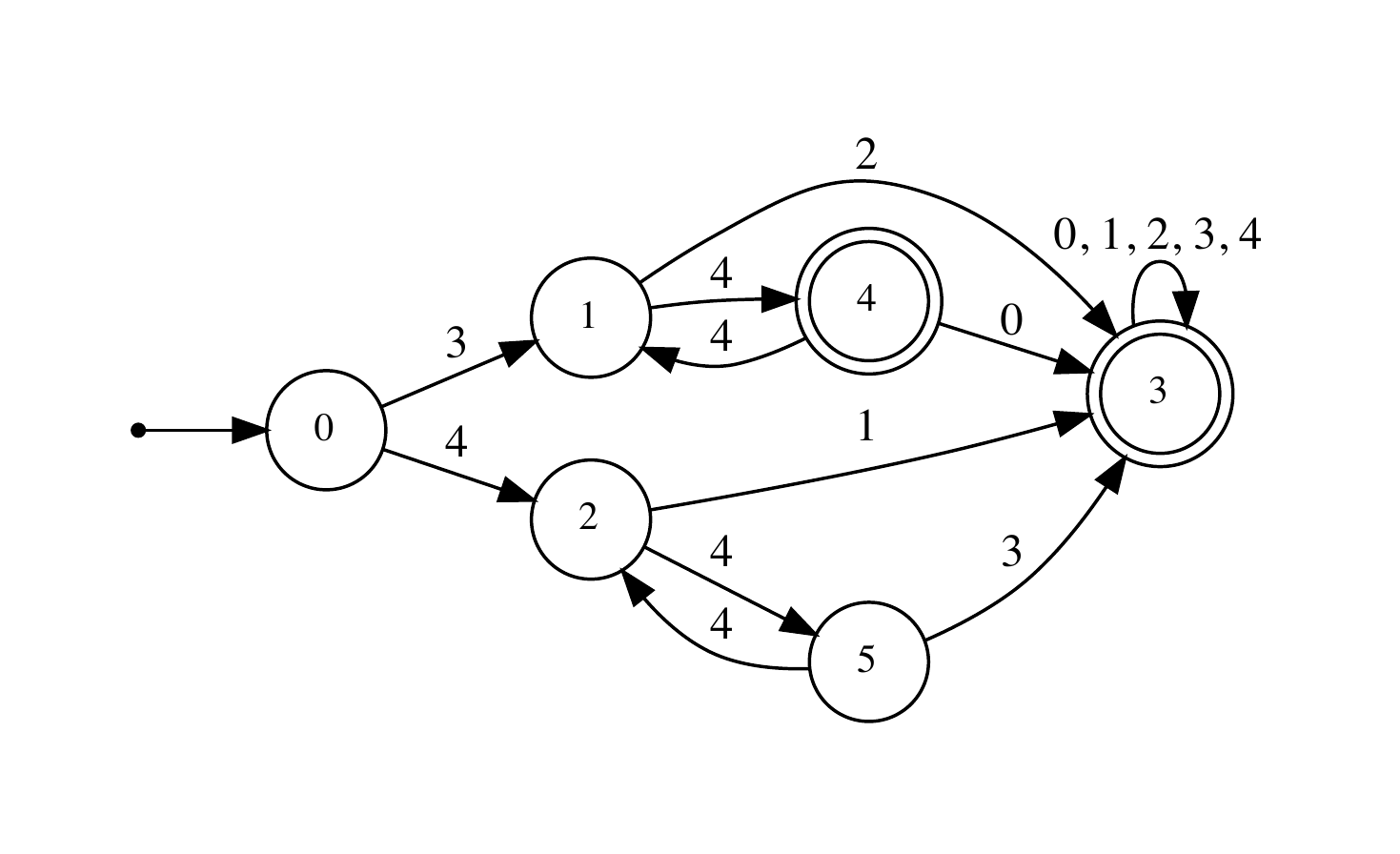}
    \caption{Automaton for the positions of the $0$'s in ${\bf m}_5$}\label{mot5mod0}
  \end{figure}
\end{proof}

One notices that ${\bf m}_3$ contains
arbitrarily large runs of $0$'s, whereas ${\bf m}_5$ does
not have this property.  We can use Walnut to determine the types of repetitions
that are present in ${\bf m}_5$, but we first need to introduce some
definitions.

Let $w = w_1w_2\cdots w_n$ be a word of length $n$ and
\emph{period} $p$; i.e., $w_i = w_{i+p}$  for $i = 0, \ldots, n-p$.
If $p$ is the smallest period of $w$, we say that the \emph{exponent}
of $w$ is $n/p$.  We also say that $w$ is an \emph{$(n/p)$-power} of
\emph{order $p$}.  Words of exponent $2$ (resp.~$3$) are called
\emph{squares} (resp.~\emph{cubes}).  If ${\bf x}$ is an infinite
sequence, we define the \emph{critical exponent} of ${\bf x}$ as
$$ \sup \{ e \in \mathbb{Q}: \text{ there is a factor of } {\bf x}
   \text{ with exponent } e \}. $$

\begin{theorem}
The sequence ${\bf m}_5$ has critical exponent $3$.  Furthermore,
the only cubes in ${\bf m}_5$ are $111$, $222$, $333$, and $444$.
\end{theorem}

\begin{proof}
We execute the Walnut commands
\begin{verbatim}
eval tmp "?lsd_5 Ei,n (n>=1) & At (t<=2*n) => MOT5[i+t]=MOT5[i+t+n]":
eval tmp "?lsd_5 Ei (n>=1) & At (t<2*n) => MOT5[i+t]=MOT5[i+t+n]":
\end{verbatim}
and note that the first outputs ``FALSE'', indicating that ${\bf m}_5$
has no factors of exponent larger than $3$, and the second produces
an automaton that only accepts $n=1$, indicating that the only cubes
in ${\bf m}_5$ have order $1$.  By inspecting a prefix of ${\bf m}_5$,
one sees that $111$, $222$, $333$, and $444$ all occur.
\end{proof}

We can also prove that every pattern of residues that appears in
${\bf m}_5$ appears infinitely often, and furthermore, we can
give a bound on when the next occurrence of a pattern will appear
in ${\bf m}_5$.  We say that a sequence ${\bf x}$ is
\emph{uniformly recurrent} if for every factor $w$ of ${\bf x}$,
there is a constant $c$ such that every occurrence of $w$ in ${\bf x}$
is followed by another occurrence of $w$ at distance at most $c$.

Note that ${\bf m}_3$ is \emph{not} uniformly recurrent.  This is
due to the presence of arbitrarily large runs of $0$'s in ${\bf m}_3$.
On the other hand, the sequence ${\bf m}_5$ exhibits rather different
behaviour.

\begin{theorem}
The sequence ${\bf m}_5$ is uniformly recurrent.  Furthermore,
if $w$ has length $n$ and occurs at position $i$ in ${\bf m}_5$,
then there is another occurrence of $w$ at some position $j$,
where $i < j \leq i+200n$.  The bound $200n$ cannot be replaced by
$200n-1$.
\end{theorem}

\begin{proof}
This is proved with the Walnut commands
\begin{verbatim}
def mot5faceq "?lsd_5 At (t<n) => (MOT5[i+t]=MOT5[j+t])":
eval tmp "?lsd_5 An (n>=1) => Ai Ej (j>i) & (j<i+200*n+1) &
    $mot5faceq(i,j,n)":
eval tmp "?lsd_5 An (n>=1) => Ai Ej (j>i) & (j<i+200*n) &
    $mot5faceq(i,j,n)":
\end{verbatim}
noting that the first \texttt{eval} command returns ``TRUE'' and
the second returns ``FALSE''.
\end{proof}

Burns \cite{Bur_E} studied ${\bf m}_p$ for $p$ between $7$
and $29$ using automata computed using the Rowland--Yassawi algorithm.
Among other things, his work suggests that depending on the value of $p$, the sequence
${\bf m}_p$ either behaves like ${\bf m}_3$, where $0$ has density $1$ (i.e.,
$p = 7, 17, 19$), or ${\bf m}_p$ behaves like ${\bf m}_5$, where $0$ has
density $<1$ (i.e., $p = 11, 13, 23, 29$).  Many of Burns' results could also be
obtained using Walnut.

\begin{problem}\label{mot_rec}Characterize the primes $p$ for which ${\bf m}_p$ is
uniformly recurrent.
\end{problem}

Indeed, based on Burns' results and the discussion in the next section,
we guess that the answer to this problem is given by the sequence
$$2, 5, 11, 13, 23, 29, 31, 37, 53, \ldots$$
of primes that do not divide any central trinomial number.  This is
sequence 
\seqnum{A113305} of \cite{OEIS}.

\section{Central trinomial coefficients}
The Motzkin numbers are closely related to the \emph{central trinomial
coefficients} $T_n$.  The usual definition of $T_n$ is as the coefficient of $x^n$
in $(1+x+x^2)^n$, but the definition
$$T_n = \sum_{k \geq 0}\binom{n}{2k}\binom{2k}{k}$$ better illustrates the connection between
these numbers and the Motzkin numbers.  The number $T_n$ is also
the constant term of $(1/x+1+x)^n$, which is the form needed for the Rowland--Zeilberger
algorithm.  Deutsch and Sagan
studied the divisibility of $T_n$ modulo primes and Noe \cite{Noe06} did the same
for generalized central trinomial numbers.

\begin{theorem}[Deutsch and Sagan]
The central trinomial coefficient $T_n$ satisfies
\[
T_n \equiv_3
\begin{cases}
1, & \text{ if } (n)_3  \text{ does not contain a } 2;\\
0, & \text{ otherwise.}
\end{cases}
\]
\end{theorem}

Deutsch and Sagan proved this by an application of Lucas' Theorem;
it is also immediate from the automaton produced by the Rowland--Zeilberger
algorithm.  As with the Motzkin numbers, the behaviour of $T_n$ modulo $5$ is
rather different from that modulo $3$.  We collect some properties below
(compare with those of ${\bf m}_5$ from the previous section).

\begin{theorem}
Let ${\bf t}_5 = (T_n \bmod 5)_{n \geq 0}$.  Then
\begin{enumerate}
\item ${\bf t}_5$ does not contain $0$ (i.e., $T_n$ is never divisible by $5$);
\item ${\bf t}_5$ has critical exponent $3$; furthermore,
the only cubes in ${\bf t}_5$ are $111$, $222$, $333$, and $444$;
\item ${\bf t}_5$ is uniformly recurrent;  Furthermore,
if $w$ has length $n$ and occurs at position $i$ in ${\bf t}_5$,
then there is another occurrence of $w$ at some position $j$,
where $i < j \leq i+200n-192$.  The constant $192$ cannot be replaced with $193$.
\item If $w$ has length $n$ and appears in ${\bf t}_5$, then
$w$ appears in the prefix of ${\bf t}_5$ of length $121n$.  The
quantity $121n$ cannot be replaced with $121n-1$.
\end{enumerate}
\end{theorem}

\begin{proof}
Properties 1)--3) can all be obtained by similar Walnut commands to those used in the previous
section for the Motzkin numbers.  We just need the automaton for ${\bf t}_5$.
The Rowland--Zeilberger algorithm gives the pleasantly symmetric automaton
in Figure~\ref{TRI5}.

  \begin{figure}[htb]
    \centering
    \includegraphics[scale=0.60]{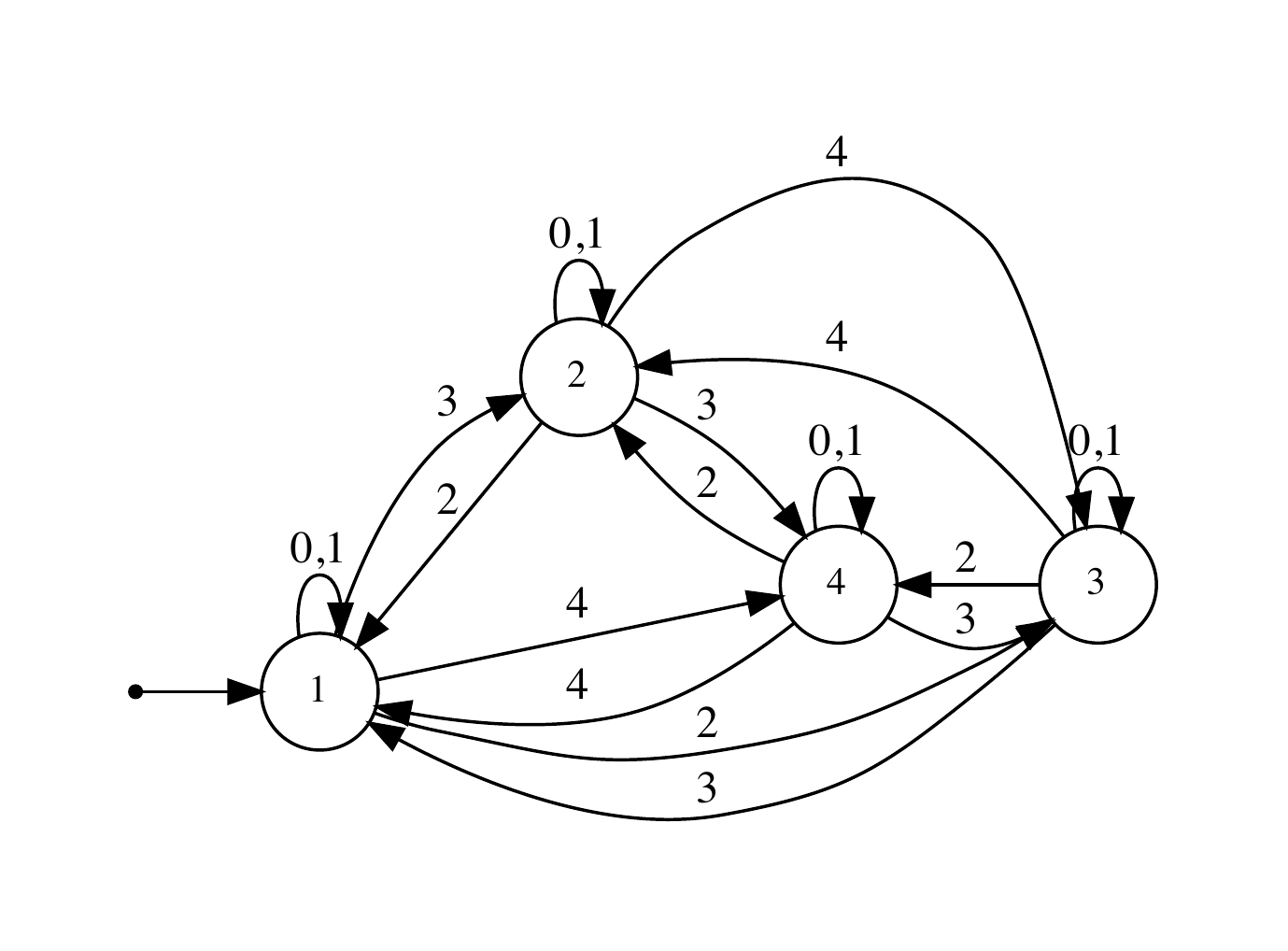}
    \caption{Automaton for $T_n \bmod 5$}\label{TRI5}
  \end{figure}

For Property 4), we use the Walnut commands
\begin{verbatim}
def pr_tri5 "?lsd_5 Aj Ei i+n<=s & At t<n => TRI5[i+t]=TRI5[j+t]":
eval tmp "?lsd_5 An $pr_tri5(n,121*n)":
eval tmp "?lsd_5 An $pr_tri5(n,121*n-1)":
\end{verbatim}
and note that the last two commands return "TRUE" and "FALSE",
respectively.
\end{proof}

We should note that in the special case of the central trinomial coefficients,
it is not necessary to resort to either the Rowland--Zeilberger or Rowland--Yassawi algorithms to compute the automaton for $T_n \bmod p$.  Using the following result of Deutsch and Sagan,
one can directly define the automaton for $T_n \bmod p$.

\begin{theorem}{(Deutsch and Sagan)}\label{tri_lucas}
Let $(n)_p = n_0n_1\cdots n_r$. Then $$T_n \equiv_p \prod_{i=0}^r T_{n_i}.$$
\end{theorem}

An immediate consequence is that $T_n$ is divisible by $p$ if and only if one of the $T_{n_i}$
is divisible by $p$.
This criterion allows one to determine the primes that do not divide any central trinomial
coefficient; i.e., those in \seqnum{A113305} of \cite{OEIS}, which we conjectured in the
previous section to be the ones that answer the question of Problem~\ref{mot_rec}.

We can also give the following sufficient condition for ${\bf t}_p = (T_n \bmod p)_{n \geq 0}$
to be uniformly recurrent.  For $i =0,\ldots,p-1$, let $\tau_i = T_i \bmod p$.

\begin{theorem}
Let $p$ be prime and let $\Sigma = \{\tau_i : i = 0,\ldots,p-1\}$.  If $\Sigma$ does not
contain $0$ but does contain a primitive root modulo $p$, then ${\bf t}_p$ is uniformly
recurrent.
\end{theorem}

\begin{proof}
Clearly the order of the product in Theorem~\ref{tri_lucas} does not matter; it follows then
that Theorem~\ref{tri_lucas} holds for the most-significant-digit-first representation
of $n$, as well as the least-significant-digit-first representation.  If we consider
Theorem~\ref{tri_lucas} with $n$ written in msd-first notation, we see that ${\bf t}_p$
is generated by iterating the morphism $f : \Sigma^* \to \Sigma^*$ defined by
$$f(\tau_i) = (\tau_i\tau_0 \bmod p)(\tau_i\tau_1 \bmod p)\cdots (\tau_i\tau_{p-1} \bmod p)$$
for $i = 0,\ldots,p-1$; i.e., ${\bf t}_p = f^\omega(1)$.

Recall that if there exists $t$ such that for every $a,b \in \Sigma$ the word $f^t(a)$
contains $b$, we say that $f$ is a \emph{primitive morphism}.  Now $\tau_0=1$, so for
$i=0,\ldots,p-1$, we can write $f(\tau_i)=\tau_i x_i$ for some word $x_i$.  It follows
that $f^p(\tau_i) = \tau_i f(x_i) f^2(x_i)\cdots f^{p-1}(x_i)$.  Furthermore,
if $0 \notin \Sigma$ and $x_0$ contains a primitive root modulo $p$, then for every $i$,
each non-zero residue modulo $p$ appears in one of $\tau_i, f(x_i), f^2(x_i), \ldots,
f^{p-1}(x_i)$.  This proves that the morphism $f$ is primitive.  A standard result
from the theory of morphic sequences states that any fixed point of a primitive morphism
is uniformly recurrent \cite[Theorem~10.9.5]{AS03}.
\end{proof}

\begin{example}
For $p=5$, we have $(T_0,T_1,T_2,T_3,T_4) = (1,1,3,7,19)$, so $(\tau_0,\tau_1,\tau_2,\tau_3,\tau_4) =
(1,1,3,2,4)$ contains the primitive root $2$.  The word
$${\bf t}_5 = 113241132433412221434423111324\cdots$$ is
uniformly recurrent and is equal to
$f^\omega(1)$, where $f$ is the morphism defined by
\begin{align*}
1 &\to 11324\\
2 &\to 22143\\
3 &\to 33412\\
4 &\to 44231.
\end{align*}
\end{example}

A computer calculation shows that for each prime $p$ appearing in the list of
initial values $2,5,11,13,\ldots,479$ of \seqnum{A113305}, the first $p$
terms of ${\bf t}_p$ always contain a primitive root modulo $p$.
Hence, each of these ${\bf t}_p$'s are uniformly recurrent.

\section{Catalan numbers}
Alter and Kubota \cite{AK73} studied the sequences ${\bf c}_p = (C_n \bmod p)_{n
  \geq 0}$, where $p$ is prime.  They proved that the runs of $0$'s in
${\bf c}_p$ have lengths
\begin{equation}\label{cp_0runs}
  \frac{p^{m+1+\delta_{3p}}-3}{2},
\end{equation}
where $\delta_{3p}$ is $1$ when $p=3$ and $0$ otherwise.  This implies, of
course, that for every prime $p$, the sequence ${\bf c}_p$ is not
uniformly recurrent.  Alter and Kubota also
proved that the blocks of non-zero values in ${\bf c}_p$ have length
$$\frac{p+3(1+2\delta_{3p})}{2}.$$
For $p=3$, Deutsch and Sagan~\cite[Theorem~5.2]{DS06} gave a complete
characterization of ${\bf c}_3$.  We can obtain a similar
characterization using Walnut.

\begin{theorem}[Deutsch and Sagan]\label{c3_0runs}
  The runs of $0$'s in ${\bf c}_3$ begin at positions $n$, where either
  $$(n)_3 \in 211^* \text{ or } (n)_3 \in 211^*0\{0,1\}^*,$$
  and have length $(3^{i+2}-3)/2$, where $i$ is the length of the
  leftmost block of $1$'s in $(n)_3$.  The blocks of non-zero values
  in ${\bf c}_3$ are given by the following:
  \begin{itemize}
  \item The block $11222$ occurs at position $0$.
  \item The block $111222$ occurs at all positions $n$ where $(n)_3 \in 222^*0w$
  for some $w \in \{0,1\}^*$ that contains an odd number of $1$'s.
  \item The block $222111$ occurs at all positions $n$ where $(n)_3 \in 222^*0w$
  for some $w \in \{0,1\}^*$ that contains an even number of $1$'s.
  \end{itemize}
\end{theorem}

\begin{proof}
  We use the automaton for ${\bf c}_3$ given in Figure~\ref{CAT3}.

  \begin{figure}[htb]
    \centering
    \includegraphics[scale=0.75]{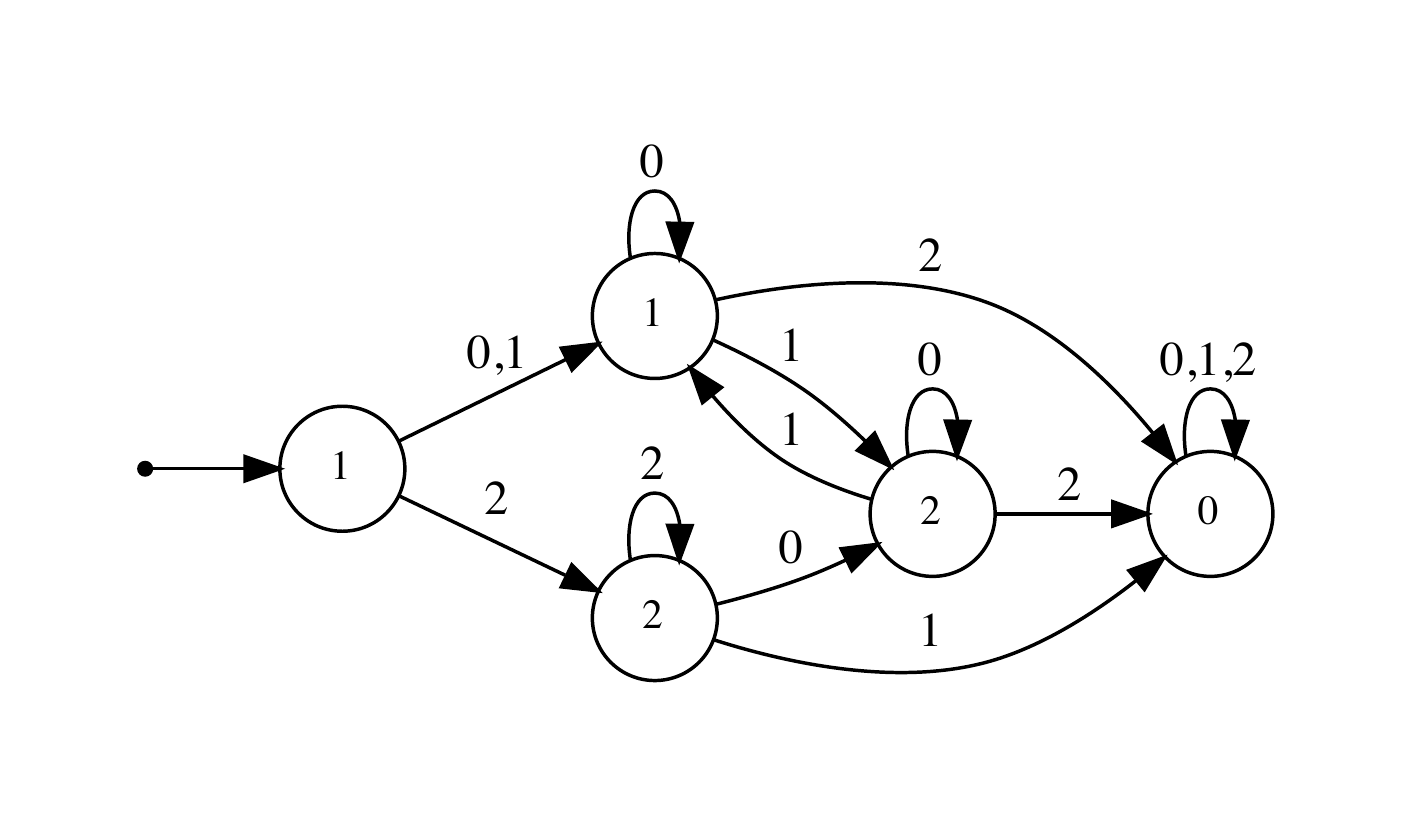}
    \caption{Automaton for $C_n \bmod 3$}\label{CAT3}
  \end{figure}

  The Walnut command
\begin{verbatim}
eval cat3max0 "?lsd_3 n>=1 & (At t<n => CAT3[i+t]=@0) &
    CAT3[i+n]!=@0 & (i=0|CAT3[i-1]!=@0)":
\end{verbatim}
  produces the automaton in Figure~\ref{cat3max0}.
  \begin{figure}[htb]
    \centering
    \includegraphics[scale=0.75]{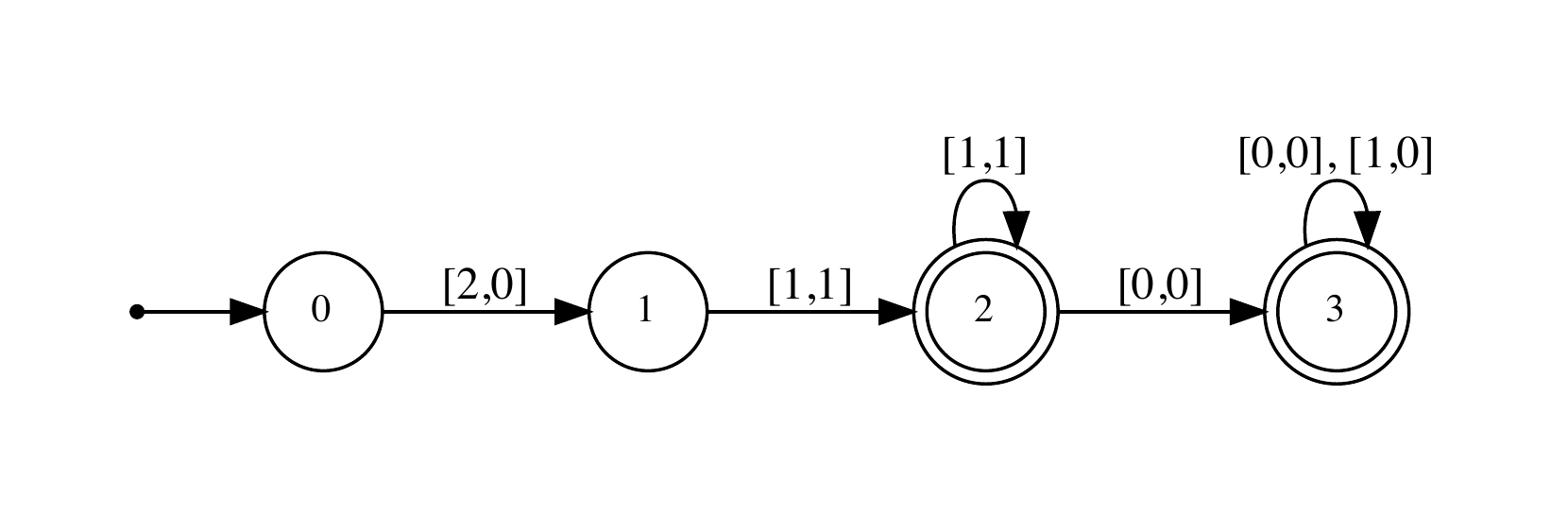}
    \caption{Automaton for runs of $0$'s in ${\bf c}_3$}\label{cat3max0}
  \end{figure}
  Examining the transition
  labels of the first component of the input gives the claimed representation for
  the starting positions of the runs of $0$'s and examining the transition
  labels of the second component gives the claimed length.
  
  For the blocks of non-zero values, we execute the Walnut commands
\begin{verbatim}
eval cat3max12 "?lsd_3 n>=1 & (At t<n => CAT3[i+t]!=@0) &
    CAT3[i+n]=@0 & (i=0|CAT3[i-1]=@0)":
eval cat3_111222 "?lsd_3 $cat3max12(i,6) & CAT3[i]=@1 &
    CAT3[i+1]=@1 & CAT3[i+2]=@1 & CAT3[i+3]=@2 &
    CAT3[i+4]=@2 & CAT3[i+5]=@2":
eval cat3_222111 "?lsd_3 $cat3max12(i,6) & CAT3[i]=@2 &
    CAT3[i+1]=@2 & CAT3[i+2]=@2 & CAT3[i+3]=@1 &
    CAT3[i+4]=@1 & CAT3[i+5]=@1":
eval cat3all12 "?lsd_3 Ai,n $cat3max12(i,n) =>
    (i=0 | $cat3_111222(i) | $cat3_222111(i))":
\end{verbatim}
to obtain the automata in Figures~\ref{cat3_111222} and \ref{cat3_222111}.
  \begin{figure}[htb]
    \centering
    \includegraphics[scale=0.75]{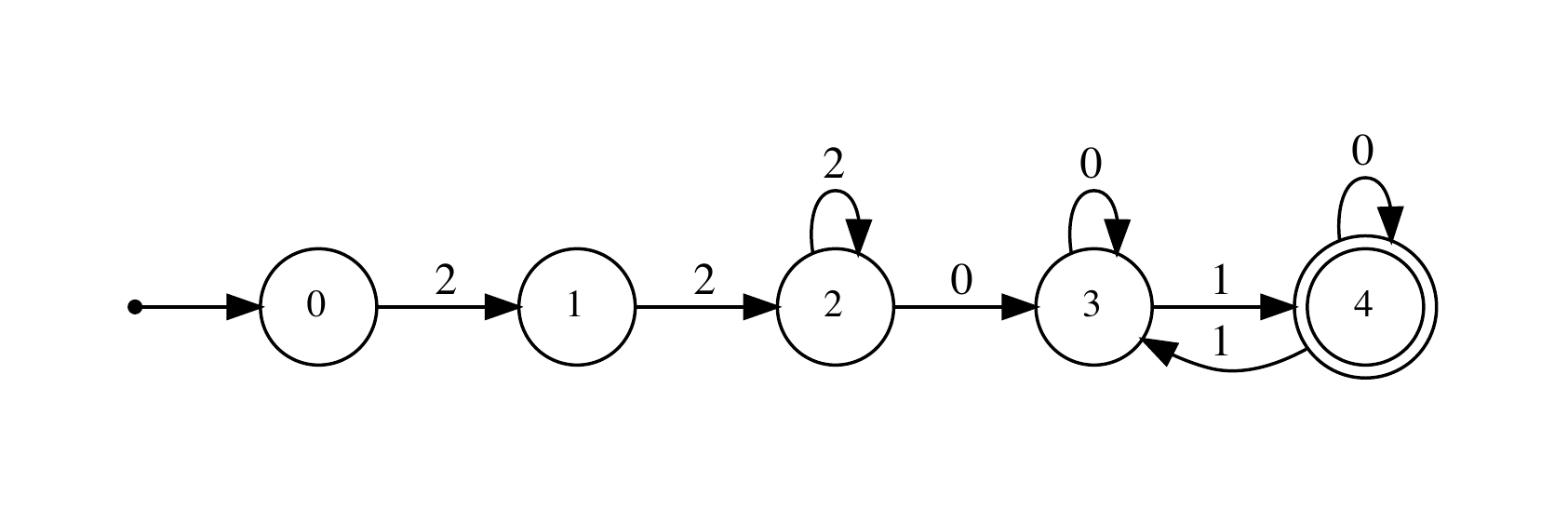}
    \caption{Automaton for blocks $111222$ in ${\bf c}_3$}\label{cat3_111222}
  \end{figure}
  \begin{figure}[htb]
    \centering
    \includegraphics[scale=0.75]{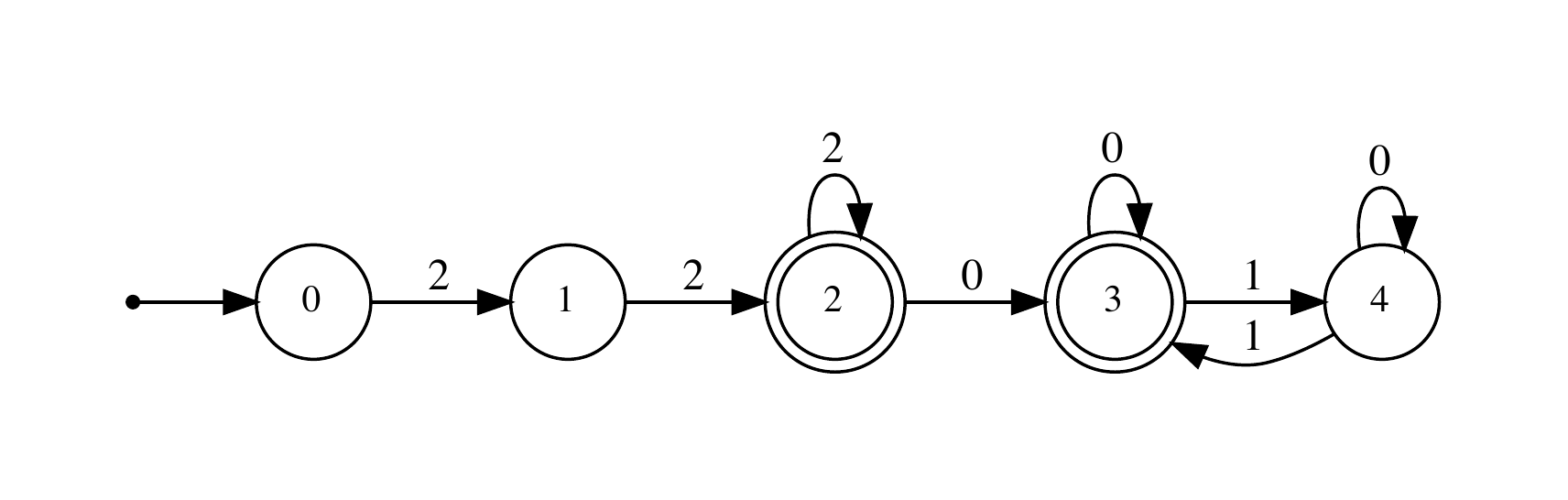}
    \caption{Automaton for blocks $222111$ in ${\bf c}_3$}\label{cat3_222111}
  \end{figure}
\end{proof}

Note that the length of the runs given in Theorem~\ref{c3_0runs} is
exactly what is given by the result of Alter and Kubota stated above
in Eq.~\eqref{cp_0runs}.

We can also perform the same calculation for $p=5$ to obtain

\begin{theorem}
 The runs of $0$'s in ${\bf c}_5$ begin at positions $n$, where either
  $$(n)_5 \in 32^* \text{ or } (n)_5 \in 32^*\{0,1\}\{0,1,2\}^*,$$
  and have length $(5^{i+2}-3)/2$, where $i$ is the length of the leftmost
  block of $2$'s in $(n)_5$.
\end{theorem}

\begin{proof}
  We use the automaton for ${\bf c}_5$ given in Figure~\ref{CAT5}.
  \begin{figure}[htb]
    \centering
    \includegraphics[scale=0.65]{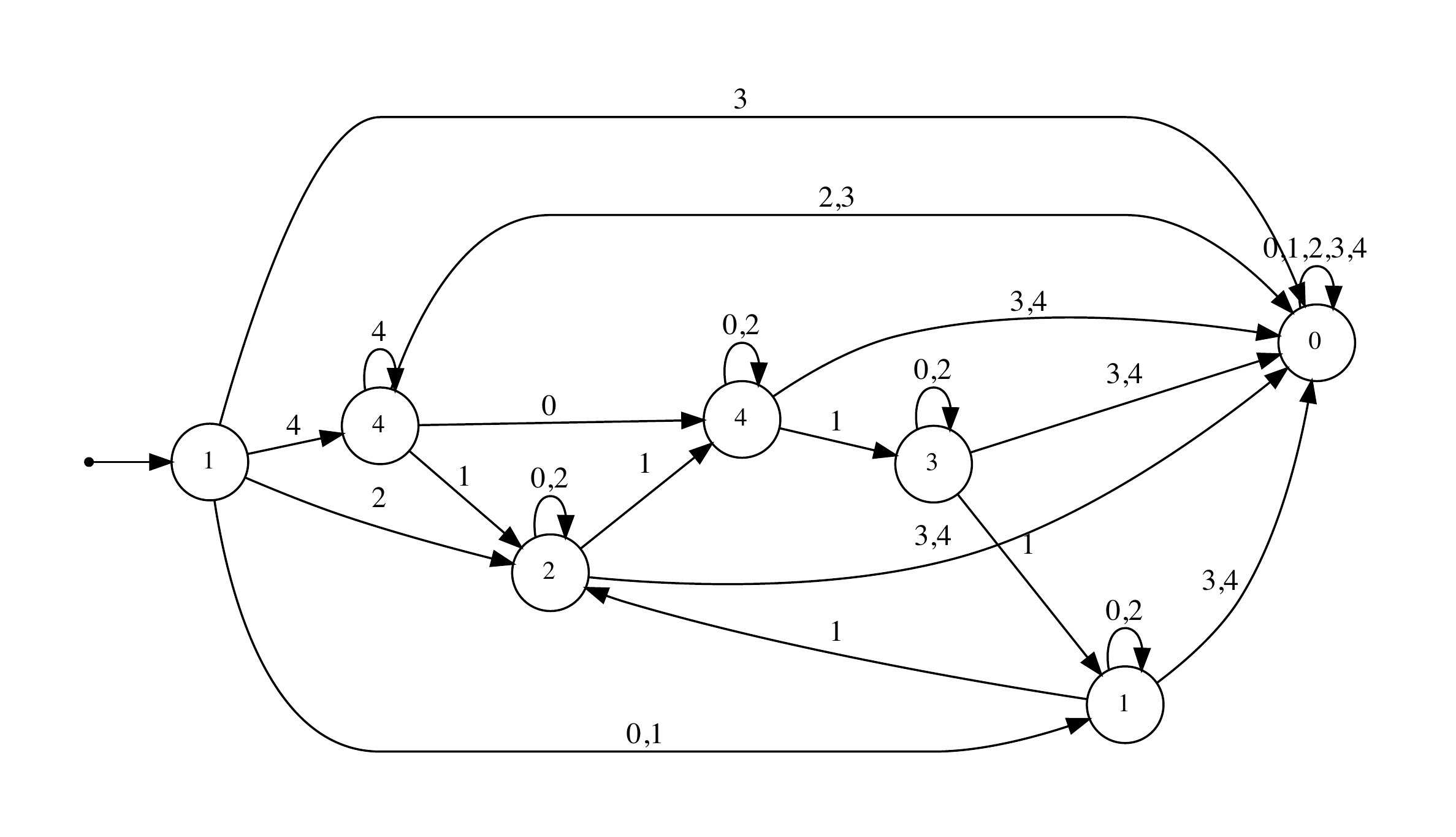}
    \caption{Automaton for $C_n \bmod 5$}\label{CAT5}
  \end{figure}
The Walnut command
\begin{verbatim}
eval cat5max0 "?lsd_5 n>=1 & (At t<n => CAT5[i+t]=@0) &
    CAT5[i+n]!=@0 & (i=0|CAT5[i-1]!=@0)":
\end{verbatim}
  produces the automaton in Figure~\ref{cat5max0}.
  \begin{figure}[htb]
    \centering
    \includegraphics[scale=0.75]{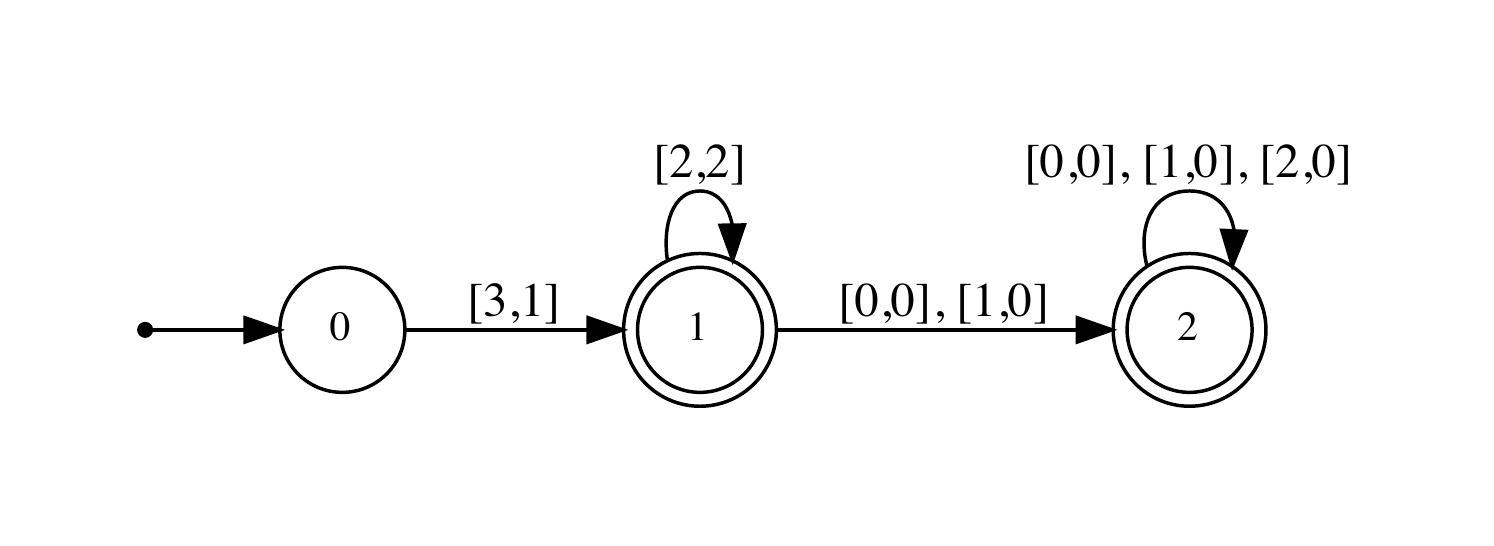}
    \caption{Automaton for runs of $0$'s in ${\bf c}_5$}\label{cat5max0}
  \end{figure}
  Examining the transition
  labels, as in the proof of Theorem~\ref{c3_0runs}, gives the result.
\end{proof}

Again, note that the lengths of the runs match what is given by
Eq.~\eqref{cp_0runs}.

\begin{theorem}
The sequence $c_5$ begins with the non-zero block $112$.  The other
non-zero blocks in $c_5$ are $1331$, $2112$, $3443$, and $4224$.
\end{theorem}

\begin{proof}
The Walnut command
\begin{verbatim}
eval cat5max1234 "?lsd_5 n>=1 & (At t<n => CAT5[i+t]!=@0) &
    CAT5[i+n]=@0 & (i=0|CAT5[i-1]=@0)":
\end{verbatim}
  produces the automaton in Figure~\ref{cat5max1234}.
  \begin{figure}[htb]
    \centering
    \includegraphics[scale=0.75]{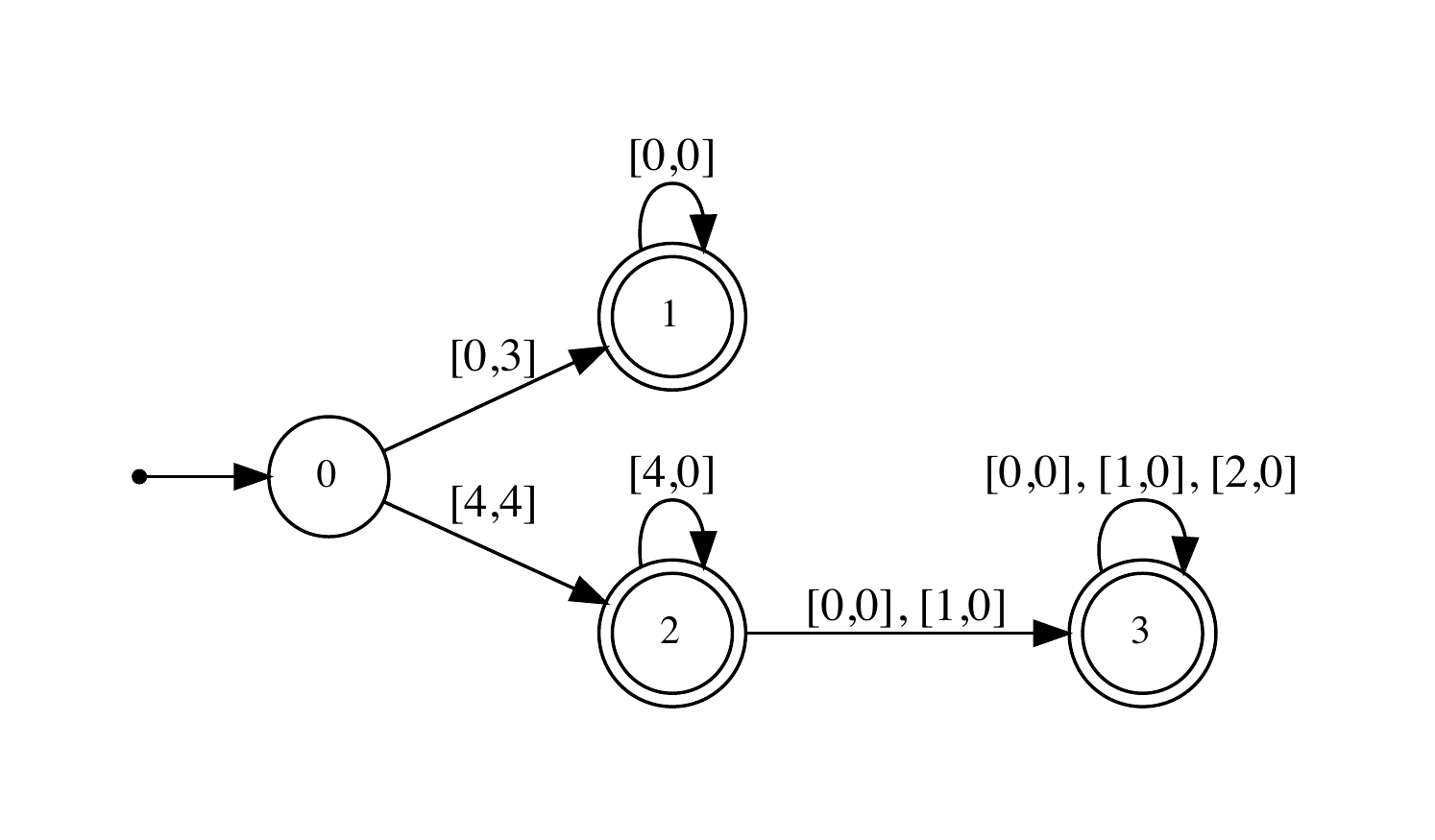}
    \caption{Automaton for non-zero blocks in ${\bf c}_5$}\label{cat5max1234}
  \end{figure}
We see that the initial non-zero block has length $3$ and all others
have length $4$.  We omit the Walnut command to verify
the values of these length $4$ blocks, but it is easy to
formulate.
\end{proof}

\section{Conclusion}
We have shown how to use Walnut to obtain automated proofs of certain
results in the literature concerning the Catalan and Motzkin numbers modulo $p$,
as well as the central trinomial coefficients modulo $p$.
We were also able to use Walnut to examine other properties of these
sequences that have not previously been explored, such as the presence (or
absence) of certain repetitive patterns and the property of being
uniformly recurrent.  We hope these results encourage other researchers to
continue to further explore these properties for other sequences.

\end{document}